\documentclass[11pt]{amsart}
\usepackage{amssymb, amsmath, amsthm, amscd, dblaccnt}
\usepackage{mathtools}
\newtheorem{proposition}{Proposition}[section]

\newtheorem{corollary}[proposition]{Corollary}
\newtheorem{theorem}[proposition]{Theorem}

\theoremstyle{definition}
\newtheorem{definition}[proposition]{Definition}
\newtheorem{example}[proposition]{Example}
\newtheorem{examples}[proposition]{Examples}

\newcommand{\thlabel}[1]{\label{th:#1}}
\newcommand{\thref}[1]{Theorem~\ref{th:#1}}
\newcommand{\selabel}[1]{\label{se:#1}}
\newcommand{\seref}[1]{Section~\ref{se:#1}}

\newcommand{\colabel}[1]{\label{co:#1}}
\newcommand{\coref}[1]{Corollary~\ref{co:#1}}

\newcommand{\exlabel}[1]{\label{ex:#1}}
\newcommand{\exref}[1]{Example~\ref{ex:#1}}

\newcommand{\eqlabel}[1]{\label{eq:#1}}
\newcommand{\equref}[1]{(\ref{eq:#1})}

\def\ot{\otimes}

\def\NN{{\mathbb N}}

\def\ZZ{{\mathbb Z}}

\newcommand{\Cc}{\mathcal{C}}

\def\*C{{}^*\hspace*{-1pt}{\Cc}}
\def\text#1{{\rm {\rm #1}}}

\input xy
\xyoption {all} \CompileMatrices

\usepackage{amssymb}
\usepackage{color,amssymb,graphicx,amscd,amsmath}
\usepackage[colorlinks,urlcolor=blue,linkcolor=blue,citecolor=blue]{hyperref}

\begin{document}

\title[Constructing Hopf braces]{Constructing Hopf braces}

\author{A. L. Agore}
\address{Simion Stoilow Institute of Mathematics of the Romanian Academy, P.O. Box 1-764, 014700 Bucharest, Romania \textbf{and} Vrije Universiteit Brussel, Pleinlaan 2, B-1050 Brussels, Belgium} \email{ana.agore@vub.be and
ana.agore@gmail.com}

\thanks{This work was  supported by a grant of Romanian Ministery of Research and Innovation, CNCS - UEFISCDI, project number PN-III-P1-1.1-TE-2016-0124, within PNCDI III. The author is a fellow of FWO (Fonds voor Wetenschappelijk Onderzoek Flanders).}



\begin{abstract}
We investigate Hopf braces, a concept  recently introduced by Angiono, Galindo and Vendramin (\cite{AGV}) in connection to the quantum Yang-Baxter equation. More precisely, we propose two methods for constructing Hopf braces. The first one uses matched pairs of Hopf algebras while the second one relies on category-theoretic tools.
\end{abstract}

\maketitle

\section*{Introduction}
The Yang-Baxter equation first appeared in the field of statistical mechanics in independent papers by Yang (\cite{yang}) and respectively Baxter (\cite{bax}). Given a vector space $V$, a linear map $R: V \ot V \to V \ot V$ is called a solution of the quantum Yang-Baxter equation if 
\begin{eqnarray}
R^{12}R^{13}R^{^{23}} = R^{23}R^{13}R^{12}\eqlabel{-10}
\end{eqnarray}
in ${\rm End}(V^{\ot 3})$, where $R^{ij}$ denotes $R$ acting on the $i$-th and respectively $j$-th components. It is well-known that \equref{-10} is equivalent to the braid equation in the sense that $R$ is a solution of the quantum Yang-Baxter equation if and only if $R \tau$ is a solution of the braid equation, where $\tau$ denotes the flip map (see for instance \cite[Proposition 114]{CMZ2002}).

Although the problem of finding solutions to the Yang-Baxter equation is wide open, the interest it generated has led to the development of many fields and even to the appearance of new ones such as the theory of quantum groups. Due to its connections with various seemingly unrelated areas of mathematics and physics (e.g., knot theory, non-commutative geometry, conformal field theory, quantum groups, integrable systems, etc.), the Yang-Baxter equation was intensively studied from many different points of view and using a broad range of techniques. 

Following an idea of Drinfel'd' (\cite{D}), set-theoretical solutions of the Yang-Baxter equation are being studied quite intensively. An important class of solutions is the non-degenerate involutive set-theoretical solutions which led to the introduction of braces by Rump (\cite{Rump}). The equivalent definition of a brace proposed in \cite{eric} opened the way for generalizing this concept to the non-commutative setting: skew-braces were introduced in  \cite{GV} as a tool for studying the non-involutive set-theoretical solutions. These newly introduced concepts allow for methods from both group theory and ring theory to be used in studying set-theoretical solutions of the Yang-Baxter equation. 
Hopf braces were recently considered in \cite{AGV} as the quantum version of skew-braces and, as expected, they provide solutions to the quantum Yang-Baxter equation. The purpose of this paper is to further investigate this new structures; more specifically we focus on constructing new examples of (cocommutative) Hopf braces. This is achieved on the one hand by using matched pairs of Hopf algebras and on the other hand by pursuing a category-theoretic path. Although the idea of using matched pairs of groups or Hopf algebras for constructing solutions of the Yang-Baxter equation has been used before (see for instance \cite{GIM}, \cite[Proposition 3.2]{AGV}) our approach is different from the existing ones.

The outline of the paper is as follows. In \seref{1} we collect some auxiliary material needed in the sequel. \seref{2} presents two methods of constructing new (cocommutative) Hopf braces from a given (cocommutative) Hopf brace which is a part of a matched pair of (cocommutative) Hopf algebras (\thref{main0}, \thref{main}). One notable consequence is \coref{11.11} which proves that any matched pair between two cocommutative Hopf algebras $A$ and $H$ gives rise to a cocommutative Hopf brace on the tensor product coalgebra $A \ot H$. Furthermore, any matched pair of groups or Lie algebras induces a cocommutative Hopf brace on the tensor product coalgebra of the corresponding group algebras, respectively the corresponding universal enveloping algebras (\exref{imp}). Several explicit examples are also presented. Finally, \seref{3} proves the completeness of the categories of Hopf braces and respectively cocommutative Hopf braces (\thref{limit}). In particular, we obtain that the category of bijective $1$-cocycles is also complete.

\section{Preliminaries}\selabel{1}

Throughout this paper, $k$ will be a field. Unless specified
otherwise, all vector spaces, tensor products, homomorphisms,
algebras, coalgebras, bialgebras, Lie algebras, Hopf algebras live over the field $k$. For a coalgebra
$C$, we will use Sweedler's $\Sigma$-notation $\Delta(c)=
c_{(1)}\ot c_{(2)}$ with suppressed summation sign. We use the classical notations for opposite and coopposite structures: $A^{op}$ is the opposite of the algebra $A$ while $C^{cop}$ stands for the coopposite of the coalgebra $C$. We refer to \cite{Sw} for further details concerning Hopf algebras and to \cite[Chapter 5]{CMZ2002} for a thorough background on the Yang-Baxter equation.

Without going into great detail, we collect here some useful notions and results which will be used in the sequel.
Recall that a \emph{Hopf brace} (\cite{AGV}) over a coalgebra $(H, \Delta, \varepsilon)$ consists of two Hopf algebra structures on $H$, denoted by $(H, \cdot, 1, \Delta, \varepsilon, S)$ and respectively $(H, \circ, 1_{\circ}, \Delta, \varepsilon, T)$, compatible in the sense that for all $x$, $y$, $z \in H$ we have:
\begin{equation}\eqlabel{1.1}
x \circ (yz) = \bigl(x_{(1)} \circ y\bigl) S(x_{(2)}) \bigl(x_{(3)} \circ z\bigl)
\end{equation}
where the multiplication of the Hopf algebra $(H, \cdot, 1, \Delta, \varepsilon, S)$ is denoted by juxtaposition. In order to keep the notation simple, we will denote the two Hopf algebra structures of a Hopf brace simply by $H$ and respectively $H_{\circ}$.
By setting $x = y = 1_{\circ}$ and respectively $x = z = 1_{\circ}$ in \equref{1.1}, it can be easily seen that in any Hopf brace we have $1 = 1_{\circ}$ (see \cite{AGV}).
A Hopf brace will be called cocommutative if its underlying coalgebra is cocommutative.

The main feature of cocommutative Hopf braces is that they provide solutions to the braid equation (\cite[Theorem 2.3, Corollary 2.4]{AGV}). More precisely, if $(H, \cdot, \circ, 1, \Delta, \varepsilon, S, T)$ is a cocommutative Hopf brace then the operator $c: H \ot H \to H \ot H$ defined below provides a solution to the braid equation:
$$
c(x \ot y) = S(x_{(1)}) (x_{(2)} \circ y_{(1)}) \ot T\bigl(S(x_{(3)})(x_{(4)} \circ y_{(2)})\bigl) \circ x_{(5)} \circ y_{(3)}
$$

\begin{examples}\exlabel{ex1.1}
$1)$ Any Hopf algebra $(H, \cdot, 1, \Delta, \varepsilon, S)$  gives rise to a Hopf brace by considering $x \circ y = xy$ for all $x$, $y \in H$. If $H$ is cocommutative, the corresponding solution of the braid equation is given by:
$$c(x \ot y) = y_{(1)} \ot S(y_{(2)}) x y_{(3)}$$

$2)$ If $(H, \cdot, 1, \Delta, \varepsilon, S)$  is a Hopf algebra such that $S^{2} = {\rm Id}_{H}$ then we can define a new multiplication on $H$ by $x \circ y = yx$, for all $x$, $y \in H$, which together with the underlying coalgebra structure of $H$ form a new Hopf algebra. Moreover, it is straightforward to check that \equref{1.1} is trivially fulfilled and thus  $(H, \cdot, \circ, 1, \Delta, \varepsilon, S, S)$ is a Hopf brace. If $H$ is cocommutative (which automatically implies $S^{2} = {\rm Id}_{H}$) then we obtain the following solution of the braid equation:
$$c(x \ot y) = S(x_{(1)}) y x_{(2)} \ot x_{(3)}$$

$3)$ An important class of examples can be obtained from group theory and respectively from Lie algebra theory. More precisely, in light of our previous example, any group algebra and respectively any universal enveloping algebra of a Lie algebra is an example of a cocommutative Hopf brace.
\end{examples}

Given two Hopf braces $(H, \cdot, \circ)$ and $(H', \cdot, \circ)$, a $k$-linear map $f$ between the two underlying vector spaces is called a \emph{morphism of Hopf braces} if both $f: H \to H^{'}$ and $f: H_{\circ} \to H^{'}_{\circ}$ are morphisms of Hopf algebras. Hopf braces (resp. cocommutative Hopf braces) together with morphisms of Hopf braces form a category which we denote by $k$-HopfBr (resp. by $k$-(co)HopfBr). In what follows $k$-CoAlg stands for the category of coalgebras.

Another important notion, related to Hopf braces, is that of a  \emph{bijective $1$-cocycle} (\cite[Definition 1.10]{AGV}). If $A$ and $H$ are two Hopf algebras such that $A$ is an $H$-module algebra, then a coalgebra isomorphism $\pi: H \to A$ is called a bijective $1$-cocycle if for any $h$, $k \in H$ we have:
$$
\pi(hk) = \pi(h_{(1)}) (h_{(2)} \rightharpoonup \pi(k)).
$$
A homomorphism between two bijective $1$-cocycles $\pi: H \to A$ and $\sigma: K \to B$ is a pair $(f,\,g)$ of Hopf algebra homomorphisms $f: H \to K$ and $g: A\to B$ such that for all $a \in A$, $h \in H$ we have:
$$\sigma f = g \pi,\qquad g(h  \rightharpoonup a) = f(h)  \rightharpoonup g(a).$$ 
The category of bijective $1$-cocycles defined above will be denoted by $\mathfrak{C}$ and its full subcategory whose objects are all bijective $1$-cocycles $\pi: H \to A$, where $A$ is a fixed Hopf algebra, by $\mathfrak{C}(A)$. 

Since the first method we use for constructing Hopf braces relies on
matched pairs of Hopf algebras and the corresponding bicrossed
product, it will be worth our while to recall these notions.
\begin{definition}
A \textit{matched pair} of Hopf algebras is a quadruple $(A, H,
\triangleleft, \triangleright)$, where $A$ and $H$ are Hopf
algebras, $\triangleleft : H \otimes A \rightarrow H$,
$\triangleright: H \otimes A \rightarrow A$ are coalgebra maps
such that $(A, \triangleright)$ is a left $H$-module coalgebra,
$(H, \triangleleft)$ is a right $A$-module coalgebra and the
following compatibilities hold for any $a$, $b\in A$, $x$, $y\in
H$:
\begin{eqnarray}
x \triangleright1_{A} &{=}& \varepsilon_{H}(x)1_{A}, \quad 1_{H}
\triangleleft a =
\varepsilon_{A}(a)1_{H} \eqlabel{mp1} \\
x \triangleright(ab) &{=}& (x_{(1)} \triangleright a_{(1)}) \bigl
( (x_{(2)}\triangleleft a_{(2)})\triangleright b \bigl)
\eqlabel{mp2} \\
(x y) \triangleleft a &{=}& \bigl( x \triangleleft (y_{(1)}
\triangleright a_{(1)}) \bigl) (y_{(2)} \triangleleft a_{(2)})
\eqlabel{mp3} \\
x_{(1)} \triangleleft a_{(1)} \otimes x_{(2)} \triangleright
a_{(2)} &{=}& x_{(2)} \triangleleft a_{(2)} \otimes x_{(1)}
\triangleright a_{(1)} \eqlabel{mp4}
\end{eqnarray}
If $(A, H, \triangleleft, \triangleright)$ is a matched pair of
Hopf algebras then the $k$-module $A\ot H$ with the tensor coalgebra structure and the multiplication defined as follows for all $a$, $c\in A$, $x$, $y\in H$:
\begin{equation}\eqlabel{0010}
(a \ot x)  (b\ot y):= a (x_{(1)}\triangleright
b_{(1)}) \ot (x_{(2)} \triangleleft b_{(2)}) y
\end{equation}
is a Hopf algebra called the \textit{bicrossed product} of $A$ and $H$ and will be denoted by $A \bowtie H$. The antipode on $A \bowtie H$ is given by:
\begin{equation}\eqlabel{antipbic}
S_{A \bowtie H} ( a \bowtie x ) = S_H (x_{(2)}) \triangleright S_A
(a_{(2)}) \, \bowtie \, S_H (x_{(1)}) \triangleleft S_A (a_{(1)})
\end{equation}
for all $a\in A$ and $x \in H$.
\end{definition}

Examples of bicrossed products include the semi-direct (smash) product as defined by Molnar (\cite{Mo}) in the cocommutative case. More precisely, assume that $(A, \triangleright)$ is a left $H$-module bialgebra such that for all $a\in A$ and $x \in H$ we have:
\begin{equation}\eqlabel{smash1}
x_{(1)} \otimes x_{(2)} \triangleright a =  x_{(2)} \otimes
x_{(1)} \triangleright a
\end{equation}
Then $(A, H,  \triangleleft, \triangleright)$ is a matched pair of Hopf algebras where $\triangleleft : H \otimes A \rightarrow B$ is the trivial action (i.e., $x \triangleleft  a = x \varepsilon(a)$ for all $a\in A$, $x \in H$) and the multiplication on the corresponding bicrossed product takes the form:
\begin{equation}\eqlabel{smash2}
(a \# x) (b \# y):= a \, (x_{(1)} \triangleright b )\,\#\,
x_{(2)} y
\end{equation}

\section{Constructing Hopf braces from matched pairs of Hopf algebras}\selabel{2}

In this section we use matched pairs of (cocommutative) Hopf algebras in order to construct (cocommutative) Hopf braces. More precisely, we obtain new Hopf braces by starting with a given Hopf brace which is part of a matched pair of Hopf algebras.  

\begin{theorem}\thlabel{main0}
Let $(A, \cdot, \circ)$ be a Hopf brace and $H$ a commutative, cocommutative Hopf algebra. If $(A, \blacktriangleright)$ is a left $H$-module algebra and $(A_{\circ}, H, \blacktriangleright, \blacktriangleleft)$ is a matched pair of Hopf algebras such that for all $a$, $a' \in A$, $x \in H$ we have:
\begin{eqnarray}
x \blacktriangleleft aa' = (x_{(1)} \blacktriangleleft a) S(x_{(2)}) (x_{(3)} \blacktriangleleft a')\eqlabel{1.4.6}
\end{eqnarray}
then $A \ot H$ together with the tensor product coalgebra and the following algebra structures:
\begin{eqnarray}
(a \ot x) (b \ot y) &=& ab \ot xy \eqlabel{1.2.5}\\
(a \ot x) \circ (b \ot y) &=& a \circ (x_{(1)}  \blacktriangleright b_{(1)}) \ot (x_{(2)}  \blacktriangleleft b_{(2)}) y \eqlabel{1.2.6}
\end{eqnarray}
is a Hopf brace.
\end{theorem}
\begin{proof}
The multiplication on $A \ot H$ defined by \equref{1.2.5} together with the tensor product of coalgebras is obviously a Hopf algebra. Furthermore, since $(A_{\circ}, H, \blacktriangleright, \blacktriangleleft)$ is a matched pair of Hopf algebras then $A \ot H$ with the multiplication defined by \equref{1.2.6} and the tensor product of coalgebras is also a Hopf algebra.
Thus, we are left to prove that \equref{1.1} holds true. Indeed, for all $a$, $b$, $c \in A$ and $x$, $y$, $z \in H$ we have:
\begin{eqnarray*}
&&\hspace*{-2mm} (a\ot x) \circ \bigl((b \ot y) (c \ot z)\bigl) = a \circ (\underline{x_{(1)} \blacktriangleright b_{(1)}c_{(1)}}) \, \ot \, (x_{(2)} \blacktriangleleft b_{(2)}c_{(2)}) yz\\
&=& \hspace*{-2mm} \underline{a \circ \bigl((x_{(1)} \blacktriangleright b_{(1)})(x_{(2)} \blacktriangleright  c_{(1)})}\bigl) \, \ot \, (x_{(3)} \blacktriangleleft b_{(2)}c_{(2)}) yz\\
&\stackrel{\equref{1.1}} {=}&\hspace*{-2mm} \bigl(a_{(1)} \circ (x_{(1)} \blacktriangleright b_{(1)})\bigl) S(a_{(2)}) \bigl(a_{(3)} \circ (x_{(2)} \blacktriangleright c_{(1)})\bigl) \, \ot\, (\underline{x_{(3)} \blacktriangleleft b_{(2)}c_{(2)}}) yz\\
&\stackrel{\equref{1.4.6}} {=}& \hspace*{-2mm} \bigl(a_{(1)} \circ (x_{(1)} \blacktriangleright b_{(1)})\bigl) S(a_{(2)}) \bigl(a_{(3)} \circ (\underline{x_{(2)}} \blacktriangleright c_{(1)})\bigl) \, \ot\, (\underline{x_{(3)}} \blacktriangleleft b_{(2)}) S(\underline{x_{(4)}}) (x_{(5)} \blacktriangleleft c_{(2)}) yz\\
&=& \hspace*{-2mm} \bigl(a_{(1)} \circ (x_{(1)} \blacktriangleright b_{(1)})\bigl) S(a_{(2)}) \bigl(a_{(3)} \circ (x_{(4)} \blacktriangleright c_{(1)})\bigl) \, \ot\, (x_{(2)} \blacktriangleleft b_{(2)}) \underline{S(x_{(3)}) (x_{(5)} \blacktriangleleft c_{(2)}) yz}\\
&=& \hspace*{-2mm} \bigl(a_{(1)} \circ (x_{(1)} \blacktriangleright b_{(1)})\bigl) S(a_{(2)}) \bigl(a_{(3)} \circ (x_{(4)} \blacktriangleright c_{(1)})\bigl) \, \ot\, (x_{(2)} \blacktriangleleft b_{(2)}) y S(x_{(3)}) (x_{(5)} \blacktriangleleft c_{(2)}) z\\
&\stackrel{\equref{1.2.5}} {=}& \hspace*{-2mm} \bigl(a_{(1)} \circ (x_{(1)(1)} \blacktriangleright b_{(1)}) \, \ot \, (x_{(1)(2)} \blacktriangleleft  b_{(2)}) y \bigl) \bigl(S(a_{(2)}) \ot S(x_{(2)})\bigl) \\
&& \hspace*{-2mm} \bigl(a_{(3)} \circ (x_{(3)(1)} \blacktriangleright c_{(1)}) \, \ot \, (x_{(3)(2)} \blacktriangleleft  c_{(2)}) z \bigl)\\
&\stackrel{\equref{1.2.6}} {=}& \hspace*{-2mm} \Bigl((a_{(1)} \ot x_{(1)}) \circ (b \ot y)\Bigl) S\bigl(a_{(2)} \ot x_{(2)}\bigl) \Bigl((a_{(3)} \ot x_{(3)}) \circ (c \ot z)\Bigl)
\end{eqnarray*}\hspace*{-2mm}
where in the second equality we used the fact that $(A, \blacktriangleright)$ is a left $H$-module algebra while the fifth and respectively the sixth equality follow from the cocommutativity and commutativity of $H$.
\end{proof}

\begin{corollary}\colabel{11.00}
Let $(A, \cdot, \circ)$ be a Hopf brace and $H$ a commutative, cocommutative Hopf algebra. If $(A, \blacktriangleright)$ is a left $H$-module algebra and $(A_{\circ}, \blacktriangleright)$ is a left $H$-module bialgebra then $A \ot H$ together with the tensor product coalgebra and the following algebra structures:
\begin{eqnarray}
(a \ot x) (b \ot y) &=& ab \ot xy \eqlabel{1.2.50}\\
(a \ot x) \circ (b \ot y) &=& a \circ (x_{(1)}  \blacktriangleright b_{(1)}) \ot x_{(2)}  y \eqlabel{1.2.60}
\end{eqnarray}
is a Hopf brace.
\end{corollary}
\begin{proof}
We apply \thref{main0} for the trivial right $A_{\circ}$-module coalgebra on $H$, i.e., $x \blacktriangleleft a = x \,\varepsilon(a)$ for all $a \in A$, $x \in H$. Since $(A_{\circ}, \blacktriangleright)$ is a left $H$-module bialgebra and \equref{smash2} is fulfilled due to the cocommutativity of $H$,
the tensor product of coalgebras $A \ot H$ with the multiplication defined by \equref{1.2.60} is a Hopf algebra.
Now the conclusion follows from \thref{main0} by noticing that in this case the compatibility condition \equref{1.4.6} is trivially fulfilled.
\end{proof}

\begin{corollary}\colabel{22.00}
Let $(A, \cdot, \circ)$ be a cocommutative Hopf brace and $H$ a commutative, cocommutative Hopf algebra. If $(H, \blacktriangleleft)$ is a right $A_{\circ}$-module bialgebra such that for all $a$, $a' \in A$, $x \in H$ we have:
\begin{eqnarray}
x \blacktriangleleft aa' = (x_{(1)} \blacktriangleleft a) S(x_{(2)}) (x_{(3)} \blacktriangleleft a')\eqlabel{1.4.6}
\end{eqnarray}
then $A \ot H$ together with the tensor product coalgebra and the following algebra structures:
\begin{eqnarray}
(a \ot x) (b \ot y) &=& ab \ot xy \eqlabel{1.2.500}\\
(a \ot x) \circ (b \ot y) &=& a \circ  b_{(1)} \ot (x  \blacktriangleleft b_{(2)}) y \eqlabel{1.2.600}
\end{eqnarray}
is a Hopf brace.
\end{corollary}
\begin{proof}
We use again \thref{main0} for the trivial left $H$-module coalgebra on $A_{\circ}$, i.e., $x \blacktriangleright a = a \,\varepsilon(x)$ for all $a \in A$, $x \in H$. As $(H, \blacktriangleleft)$ is a right $A_{\circ}$-module bialgebra then the tensor product of coalgebras $A \ot H$ with the multiplication defined by \equref{1.2.60} is a Hopf algebra. Indeed, it is straightforward to see that the compatibility conditions \equref{mp1}-\equref{mp4} hold; for instance \equref{mp4} is trivially fulfilled due to the cocommutativity of $A$. The conclusion now follows from  \thref{main0}.
\end{proof}

\begin{example}
Consider the cyclic groups $C_{3}$ and $C_{6}$ and let $a$, respectively $b$, denote the generators of the aforementioned groups. Then $(k[C_{3}], k[C_{6}], \triangleright, \triangleleft)$ is a matched pair of cocommutative Hopf algebras, where:
\begin{eqnarray*}
&&\hspace*{-7mm} 1 \triangleright a^{j} = a^{j},\,\,\, b^{i} \triangleright 1 = 1,\,\,\,
b^{i} \triangleright a = \left \{\begin{array}{lll}
a,\, \mbox {if i is even}\\
a^{2},\, \mbox {if i is odd}
\end{array},\right.\,
b^{i} \triangleright a^{2} = \left \{\begin{array}{lll}
a^{2},\, \mbox {if i is even}\\
a,\, \mbox {if i is odd}
\end{array},\right. \\
&&\hspace*{-7mm} b^{i} \triangleleft 1 = b^{i},\,\,\, 1 \triangleleft a^{j} = 1,\,\,\,
b^{i} \triangleleft a = \left \{\begin{array}{lll}
b^{i},\, \mbox {if i is even}\\
b^{i+2},\, \mbox {if i is odd}
\end{array},\right.\,
b^{i} \triangleleft a^{2} = \left \{\begin{array}{lll}
b^{i},\, \mbox {if i is even}\\
b^{i+4},\, \mbox {if i is odd}
\end{array},\right.
\end{eqnarray*}
for all $i \in \{1,\, 2, \cdots, 5\}$, $j \in \{1,\, 2\}$. Moreover it can be easily seen by a straightforward computation that $\triangleleft: k[C_{6}] \ot k[C_{3}] \to k[C_{6}] $ defined above satisfies compatibility \equref{1.4.6} which in this case comes down to proving that for all $g \in C_{6}$ and $h$, $h' \in C_{3}$ we have $g \triangleleft hh' = (g \triangleleft h) g^{-1} (g \triangleleft h')$.
Indeed, for instance we have:
\begin{eqnarray*}
&& b^{2k+1} \triangleleft a^{2} = b^{2k+5} = b^{2k+3}b^{-2k-1}b^{2k+3} = (b^{2k+1} \triangleleft a) b^{-2k-1} (b^{2k+1} \triangleleft a)\\
&& b^{2k+1} \triangleleft a^{3} = b^{2k+7} = b^{2k+1} = b^{2k+3}b^{-2k-1}b^{2k+5} = (b^{2k+1} \triangleleft a) b^{-2k-1} (b^{2k+1} \triangleleft a^{2})
\end{eqnarray*}
all $k \in \NN$.
Therefore, we have a cocommutative Hopf brace on $k[C_{3}] \ot k[C_{6}]$ as in \thref{main0}.
\end{example}

\begin{theorem}\thlabel{main}
Let $A$ be a Hopf algebra and $(H, \cdot, \circ)$ a cocommutative Hopf brace. If $(A, H, \triangleright, \triangleleft)$ is matched pair of Hopf algebras and $(A, \blacktriangleright)$ is a left $H_{\circ}$-module bialgebra such that for all $a \in A$, $x$, $x' \in H$ we have:
\begin{eqnarray}
x \blacktriangleright (x' \triangleright a) &=& \bigl[(x_{(1)} \circ x') S(x_{(2)}) \triangleright (x_{(3)} \blacktriangleright a)\bigl] \eqlabel{2.1}\\
x \circ (x' \triangleleft a) &=& \bigl[ (x_{(1)} \circ x') S(x_{(2)}) \triangleleft (x_{(3)} \blacktriangleright a)\bigl] x_{(4)}\eqlabel{2.2}
\end{eqnarray}
then $A \ot H$ together with the tensor product coalgebra and the following algebra structures:
\begin{eqnarray}
(a \ot x) (b \ot y) &=& a(x_{(1)}  \triangleright b_{(1)}) \ot (x_{(2)}  \triangleleft b_{(2)})y \eqlabel{2.5}\\
(a \ot x) \circ (b \ot y) &=& a(x_{(1)}  \blacktriangleright b) \ot x_{(2)}  \circ y \eqlabel{2.6}
\end{eqnarray}
is a Hopf brace.
\end{theorem}

\begin{proof}
Since $(A, H, \triangleright, \triangleleft)$ is a matched pair of Hopf algebras then $A \ot H$ with the multiplication defined by \equref{2.5} and the tensor product of coalgebras is a Hopf algebra with antipode given by \equref{antipbic}. Furthermore, as $(A, \blacktriangleright)$ is a left $H_{\circ}$-module bialgebra then the tensor product of coalgebras $A \ot H$ with the multiplication defined by \equref{2.6} is again a Hopf algebra. Indeed, notice that  \equref{smash1} is trivially fulfilled by the cocommutativity assumption on $H$. Therefore, we are left to prove that \equref{1.1} holds true. Indeed, for all $a$, $b$, $c \in A$ and $x$, $y$, $z \in B$ we have:
\begin{eqnarray*}
&&\hspace*{-10mm} \bigl[(a_{(1)} \ot x_{(1)}) \circ (b \ot y)\bigl] S\bigl(a_{(2)} \ot x_{(2)}\bigl) \bigl[(a_{(3)} \ot x_{(3)}) \circ (c \ot z)\bigl]  \\
&& \hspace*{-10mm}= a_{(1)} (x_{(1)} \blacktriangleright b) \underline{\bigl[(x_{(2)} \circ y_{(1)}) S(x_{(7)}) \triangleright S(a_{(4)})\bigl] \Bigl[\Bigl((x_{(3)} \circ y_{(2)}) S(x_{(6)})\triangleleft S(a_{(3)})\Bigl) \triangleright a_{(5)} (x_{(8)} \blacktriangleright c_{(1)}) \Bigl]}\\
&& \hspace*{-10mm} \ot \,  \Bigl[\Bigl((x_{(4)} \circ y_{(3)}) S(x_{(5)})\triangleleft S(a_{(2)})\Bigl) \triangleleft \, a_{(6)} (x_{(9)} \blacktriangleright c_{(2)}) \Bigl] (x_{(10)} \circ z)\\
&& \hspace*{-10mm}\stackrel{\equref{mp2}} {=} a_{(1)} (x_{(1)} \blacktriangleright b) \Bigl[(x_{(2)} \circ y_{(1)}) S(x_{(5)}) \triangleright  \underline{S(a_{(3)})a_{(4)}} (x_{(6)} \blacktriangleright c_{(1)}) \Bigl] \\
&& \hspace*{-10mm} \ot \,  \Bigl[\Bigl((x_{(3)} \circ y_{(2)}) S(x_{(4)})\triangleleft S(a_{(2)})\Bigl) \triangleleft \, a_{(5)} (x_{(7)} \blacktriangleright c_{(2)}) \Bigl] (x_{(8)} \circ z)\\
&& \hspace*{-10mm} = a_{(1)} (x_{(1)} \blacktriangleright b) \Bigl[(x_{(2)} \circ y_{(1)}) S(x_{(5)}) \triangleright (x_{(6)} \blacktriangleright c_{(1)}) \Bigl] \\
&& \hspace*{-10mm} \ot \,  \Bigl[\underline{\Bigl((x_{(3)} \circ y_{(2)}) S(x_{(4)})\triangleleft S(a_{(2)})\Bigl) \triangleleft \, a_{(3)} (x_{(7)} \blacktriangleright c_{(2)})} \Bigl] (x_{(8)} \circ z)\\
&& \hspace*{-10mm} = a_{(1)} (x_{(1)} \blacktriangleright b) \Bigl[(x_{(2)} \circ y_{(1)}) S(x_{(5)}) \triangleright (x_{(6)} \blacktriangleright c_{(1)}) \Bigl] \\
&& \hspace*{-10mm} \ot \,  \Bigl[\Bigl((x_{(3)} \circ y_{(2)}) S(x_{(4)}) \Bigl)\triangleleft \Bigl(\underline{S(a_{(2)}) a_{(3)}} (x_{(7)} \blacktriangleright c_{(2)})\Bigl) \Bigl]  (x_{(8)} \circ z)\\
&& \hspace*{-10mm} = a (x_{(1)} \blacktriangleright b) \Bigl[(x_{(2)} \circ y_{(1)}) S(\underline{x_{(5)}}) \triangleright (\underline{x_{(6)}} \blacktriangleright c_{(1)}) \Bigl]\ot \,  \Bigl[\Bigl((\underline{x_{(3)}} \circ y_{(2)}) S(\underline{x_{(4)}}) \Bigl)\triangleleft (x_{(7)} \blacktriangleright c_{(2)})\Bigl]  \\
&& \hspace*{-10mm} (x_{(8)} \circ z)\\
&& \hspace*{-10mm} = a (x_{(1)} \blacktriangleright b) \Bigl[(x_{(2)} \circ y_{(1)}) S(x_{(3)}) \triangleright (x_{(4)} \blacktriangleright c_{(1)}) \Bigl]\ot \,  \Bigl[\Bigl((x_{(5)} \circ y_{(2)}) S(x_{(6)}) \Bigl)\triangleleft (x_{(7)} \blacktriangleright c_{(2)})\Bigl] \\
&& \hspace*{-10mm}(x_{(8)} \circ z)\\
&& \hspace*{-10mm}=  a (x_{(1)} \blacktriangleright b) \Bigl[\bigl(x_{(2)} \circ y_{(1)}\bigl) S(x_{(3)})  \triangleright \bigl(x_{(4)} \blacktriangleright c_{(1)}\bigl) \Bigl] \ot \,\underline{\Bigl[\bigl(x_{(5)} \circ y_{(2)}\bigl) S(x_{(6)}) \triangleleft \bigl(x_{(7)}  \blacktriangleright c_{(2)} \bigl) \Bigl]}\\
&& \hspace*{-10mm} \underline{x_{(8)}} S(x_{(9)}) (x_{(10)} \circ z)\\
&& \hspace*{-10mm}\stackrel{\equref{2.2}} {=} a (x_{(1)} \blacktriangleright b) \Bigl[\bigl(x_{(2)} \circ y_{(1)}\bigl) S(x_{(3)})  \triangleright \bigl(x_{(4)} \blacktriangleright c_{(1)}\bigl) \Bigl] \ot \, \underline{\Bigl[x_{(5)} \circ \bigl(y_{(2)} \triangleleft c_{(2)}\bigl) \bigl] S(x_{(6)}) \bigl(x_{(7)} \circ z\bigl)} \\
&& \hspace*{-10mm}\stackrel{\equref{1.1}} {=} \,\,a (x_{(1)} \blacktriangleright b) \underline{\Bigl[\bigl(x_{(2)} \circ y_{(1)}\bigl) S(x_{(3)})  \triangleright \bigl(x_{(4)} \blacktriangleright c_{(1)}\bigl) \Bigl]} \ot \, x_{(5)} \circ \bigl[(y_{(2)} \triangleleft c_{(2)}) z\bigl] \\
&& \hspace*{-10mm}\stackrel{\equref{2.1}} {=} a \underline{(x_{(1)} \blacktriangleright b) \bigl[x_{(2)} \blacktriangleright (y_{(1)} \triangleright c_{(1)})} \bigl]\,\, \ot \, \,x_{(3)} \circ \bigl[(y_{(2)} \triangleleft c_{(2)}) z\bigl] \\
&& \hspace*{-9mm} = \,\,a \bigl[x_{(1)} \blacktriangleright b (y_{(1)} \triangleright c_{(1)}) \bigl]\,\, \ot \,\, x_{(2)} \circ \bigl[(y_{(2)} \triangleleft c_{(2)}) z\bigl]\\
&& \hspace*{-9mm} = \,\,(a\ot x) \circ \bigl((b \ot y) (c \ot z)\bigl)
\end{eqnarray*}
where the sixth equality follows by the cocommutativity of $H$ while in the eleventh equality we used the fact that  $(A, \blacktriangleright)$ is a left $H_{\circ}$-module algebra.
\end{proof}

\begin{corollary}\colabel{11.11}
If $(A, H, \triangleright, \triangleleft)$ is a matched pair of Hopf algebras with $H$ cocommutative then the tensor product coalgebra $A \ot H$ with the following algebra structures:
\begin{eqnarray*}
(a \ot x) (b \ot y) &=& \bigl(a(x_{(1)}  \triangleright b_{(1)}) \ot (x_{(2)}  \triangleleft b_{(2)})y \bigl)\\
(a \ot x) \circ (b \ot y) &=& \bigl(a b \ot yx \bigl)
\end{eqnarray*}
is a Hopf brace.
\end{corollary}
\begin{proof}
Indeed, this is an immediate consequence of \thref{main} by seeing $H$ as a cocommutative Hopf brace with $x \circ x' = x'x$ (see \exref{ex1.1}, 2)) and considering the left $H_{\circ}$-module bialgebra on $A$ to be the trivial one, i.e., $x \blacktriangleright a = \varepsilon(x) \,a$ for all $a \in A$, $x \in H$. It is now straightforward to see that in this case the compatibility conditions \equref{2.1} and \equref{2.2} are trivially fulfilled and the desired conclusion follows.
\end{proof}

\begin{example}\exlabel{imp}
\coref{11.11} allows for a plethora of examples of cocommutative
Hopf braces obtained from matched pairs of groups
(\cite[Definition IX.1.1]{Kassel}) and respectively matched pairs
of Lie algebras (\cite[Definition 8.3.1]{majid2}). More precisely,
any matched pair of groups $(G, H, \rightharpoonup,
\leftharpoonup)$ extends uniquely to a matched pair of cocommutative Hopf
algebras between the corresponding group Hopf algebras  $(k[G],
k[H], \rightharpoonup, \leftharpoonup)$ (\cite[Example 1, pg.
207]{Kassel}. Similarly, any matched pair of Lie algebras $(\mathfrak{g},
\mathfrak{h}, \triangleright, \triangleleft)$ can be uniquely
extended to a matched pair of cocommutative Hopf algebras between the
corresponding universal enveloping algebras $(U(\mathfrak{g}),
U(\mathfrak{h}), \triangleright, \triangleleft)$.

Since both the group Hopf algebras and the enveloping universal
algebras are cocommutative Hopf algebras, we can conclude that any
matched pair of groups gives rise to a cocommutative Hopf brace on
$k[G] \ot k[H]$ and any matched pair of Lie algebras gives rise to
a cocommutative Hopf brace on $U(\mathfrak{g})\ot U(\mathfrak{h})$
as in \coref{11.11}.
\end{example}

\begin{example}
Consider $k[T]$ to be the polynomial Hopf algebra and let $H$ be the Hopf algebra generated by $X$, $Y$ and $Z$ subject to the following relations:
$$XY-YX=Z,\,\, XZ=ZX,\,\, YZ=ZY,$$
and the coalgebra structure given by:
\begin{eqnarray*}
&&\Delta(X) = X \ot 1 + 1\ot X,\,\, \Delta(Y) = Y \ot 1 + 1\ot Y,\,\, \Delta(Z) = Z \ot 1 + 1\ot Z,\\
&&\hspace*{40mm}\varepsilon(X) = \varepsilon(Y) = \varepsilon(Z) = 0.
\end{eqnarray*}
It can be easily seen that $k[T]$ is the universal enveloping algebra of the $1$-dimensional abelian Lie algebra while $H$ is the universal enveloping algebra of the $3$-dimensional Heisenberg Lie algebra.
Then, for any $\alpha$, $\beta \in k$ we have a matched pair of cocommutative Hopf algebras $(k[T], H, \triangleright, \triangleleft)$ defined as follows for all $n \geq 1$:
\begin{eqnarray*}
&& X \triangleright T^{n} = Z \triangleright T^{n} = 0,\,\,\, Y \triangleright T^{n} =  \alpha \sum_{i=0}^{n-1}T^{n-i} (T - \alpha \beta)^{i},\\   
&& X^{n} \triangleleft T = n \beta X^{n-1} Z,\,\,\, Y^{n} \triangleleft T = (Z - \alpha \beta Y) \sum_{i=0}^{n-1}Y^{n-1-i} (Y + \alpha)^{i},\,\,\,  Y^{n} \triangleleft T =0
\end{eqnarray*}
Using \coref{11.11} we obtain a cocommutative Hopf brace on $k[T] \ot H$.
\end{example}

\begin{corollary}\colabel{22.22}
Let $A$, $H$ be two Hopf algebras with $H$ cocommutative. If $(A, \blacktriangleright)$ is a left $H^{op}$-module bialgebra and $(A, \triangleright)$ is a left $H$-module bialgebra such that for all $a \in A$, $x$, $x' \in H$ we have:
\begin{eqnarray}
x \blacktriangleright (x' \triangleright a) = x' \triangleright (x \blacktriangleright a)\eqlabel{1.1.1}
\end{eqnarray}
then the tensor product coalgebra $A \ot H$ with the following algebra structures:
\begin{eqnarray*}
(a \ot x) (b \ot y) &=& \bigl(a(x_{(1)}  \triangleright b) \ot x_{(2)} y \bigl)\\
(a \ot x) \circ (b \ot y) &=& \bigl(a(x_{(1)}  \blacktriangleright b) \ot y x_{(2)} \bigl)
\end{eqnarray*}
is a Hopf brace.
\end{corollary}
\begin{proof}
This can be easily derived from \thref{main} by seeing $H$ as a cocommutative Hopf brace with $x \circ x' = x'x$ (see \exref{ex1.1}, 2)) and considering the right $A$-module coalgebra structure on $H$ to be trivial, i.e., $x  \triangleleft a = x \varepsilon(a)$ for all $a \in A$, $x \in H$. In this case the compatibility condition \equref{2.2} is trivially fulfilled while \equref{2.1} amounts to \equref{1.1.1} as desired.
\end{proof}

\begin{example}
Consider $k[C]$ to be the group Hopf algebra of the infinite
cyclic group in multiplicative notation and let $d$ be a generator
of $C$. If $k[X]$ denotes the polynomial Hopf algebra then given any $\alpha$, $\beta \in k^{*}$ it is easy to see that
$\triangleright$, $\blacktriangleright: k[C] \ot k[X] \to k[X]$
defined below are left $kC$-module bialgebras:
\begin{eqnarray*}
d^{i}  \triangleright X^{t} = \alpha^{it} X^{t},\quad d^{i}
\blacktriangleright X^{t} = \beta^{it} X^{t}, \quad i \in \ZZ,\, t
\in \NN
\end{eqnarray*}
and, moreover, \equref{1.1.1} is fulfilled. Thus, we obtain a
cocommutative Hopf brace on $k[X] \ot k[C]$ as in \coref{22.22}.
\end{example}

\begin{example}
Let $k$ be a field of characteristic $\neq 2$. $C_n$ denotes the cyclic group of order $n$ generated
by $c$ and $H_{4}$ is the Sweedler's $4$-dimensional Hopf algebra having
$\{1, \, g, \, x, \, gx \}$ as a basis subject to the relations:
$$
g^{2} = 1, \quad x^{2} = 0, \quad x g = -g x
$$
with the coalgebra structure and antipode given by:
$$
\Delta(g) = g \otimes g, \quad \Delta(x) = x \otimes 1 + g \otimes
x, \, \varepsilon(g) = 1, \quad \varepsilon(x) = 0, \quad S(g) = g,
\quad S(x) = -gx.
$$
If $\omega$, $\lambda$ are $n$-th roots on unity in $k$, then $(H_{4}, \triangleright)$ and respectively $(H_{4}, \blacktriangleright)$ are left $k[C_{n}]$-module bialgebras (\cite[Proposition 5.3]{abm}) where $\triangleright$, $\blacktriangleright : k[C_{n}] \ot H_{4} \to H_{4}$ are defined as follows:
\begin{eqnarray*}
c^{i}  \triangleright g = g, \quad c^{i}  \triangleright x = \omega^{i} x,\quad c^{i}  \blacktriangleright g = g, \quad c^{i}  \blacktriangleright x = \lambda^{i} x,
\end{eqnarray*}
for all $i \in \NN$. Moreover, it can be easily seen that \equref{1.1.1} is fulfilled for the two left $k[C_{n}]$-module bialgebras defined above. Therefore, we have a Hopf brace on $H_{4} \ot k[C_{n}]$ as in \coref{22.22}.
\end{example}

\begin{corollary}\colabel{33.33}
Let $A$, $H$ be two Hopf algebras with $H$ cocommutative and consider $(A, \blacktriangleright)$ a left $H^{op}$-module bialgebra and $(H, \triangleleft)$ a right $A$-module bialgebra such that for all $a \in A$, $x$, $x' \in B$ we have:
\begin{eqnarray}
(x' \triangleleft a) x = \bigl(x' \triangleleft (x_{(1)} \blacktriangleright a)\bigl) x_{(2)}\eqlabel{2.2.2}
\end{eqnarray}
Then the tensor product coalgebra $A \ot H$ with the following algebra structures:
\begin{eqnarray*}
(a \ot x) (b \ot y) &=& \bigl(a b_{(1)}  \ot (x \triangleleft b_{(2)}) y \bigl)\\
(a \ot x) \circ (b \ot y) &=& \bigl(a(x_{(1)}  \blacktriangleright b) \ot y x_{(2)} \bigl)
\end{eqnarray*}
is a Hopf brace.
\end{corollary}
\begin{proof}
We apply again \thref{main}; $H$ will be seen as a cocommutative Hopf brace with $x \circ x' = x'x$ (see \exref{ex1.1}, 2)) and the left $A$-module coalgebra structure on $H$ will be the trivial one, i.e., $a \triangleright x = x \varepsilon(a)$ for all $a \in A$, $x \in H$. Then \equref{2.1} is trivially fulfilled while  \equref{2.2} comes down to  \equref{2.2.2}.
\end{proof}

We end this section with some generic examples of (cocommutative) Hopf braces and the corresponding solutions of the braid equation.

\begin{examples}\exlabel{gen}
$1)$ Let $A$, $H$ be two Hopf algebras with $H$ cocommutative. If $(A, \triangleright)$ is a left $H$-module bialgebra then by \coref{22.22} the tensor product coalgebra $A \ot H$ with the following algebra structures:
\begin{eqnarray*}
(a \ot x) (b \ot y) = \bigl(a(x_{(1)}  \triangleright b) \ot x_{(2)} y \bigl), \qquad (a \ot x) \circ (b \ot y) = \bigl(a b \ot yx \bigl)
\end{eqnarray*}
is a Hopf brace. If $A$ is cocommutative as well we obtain a cocommutative Hopf brace on $A \ot H$ and a solution of the braid equation given by:
\begin{eqnarray*}
&& \hspace*{43mm}  c \Bigl((a \ot x) \ot (b \ot y) \Bigl) = \\
&&\Bigl(S(x_{(2)}  \triangleright b_{(1)}) \ot S(x_{(1)}) y_{(1)} x_{(3)} \Bigl) \ot \Bigl(T\bigl(S(x_{(5)}  \triangleright b_{(2)})\bigl) a_{(5)} b_{(3)} \ot T\bigl(S(x_{(4)}) \bigl)\Bigl)
\end{eqnarray*}

$2)$ Let $A$, $H$ be two Hopf algebras with $H$ cocommutative. If $(A, \blacktriangleright)$ is a left $H^{op}$-module bialgebra then by \coref{33.33} the tensor product coalgebra $A \ot H$ with the following algebra structures:
\begin{eqnarray*}
(a \ot x) (b \ot y) = \bigl(a b \ot x y \bigl), \qquad (a \ot x) \circ (b \ot y) = \bigl(a(x_{(1)}  \blacktriangleright b) \ot y x_{(2)} \bigl)
\end{eqnarray*}
is a Hopf brace. If $A$ is cocommutative as well we obtain a cocommutative Hopf brace on $A \ot H$ and a solution of the braid equation given by:
\begin{eqnarray*}
&& c \Bigl((a \ot x) \ot (b \ot y) \Bigl) =  \Bigl(x_{(2)}  \blacktriangleright b_{(1)} \ot S(x_{(1)}) y_{(1)} x_{(3)} \Bigl) \ot \\
&&\Bigl(T(x_{(6)}) T(y_{(2)}) T\bigl(S(x_{(4)})\bigl) \blacktriangleright T(x_{(7)} \blacktriangleright b_{(2)}) a (x_{(8)} \blacktriangleright b_{(3)})  \ot T\bigl(S(x_{(5)}) \bigl)\Bigl)
\end{eqnarray*}

$3)$ Let $A$, $H$ be two Hopf algebras with $H$ cocommutative. If $(H, \triangleleft)$ is a right $A$-module bialgebra then by \coref{33.33} the tensor product coalgebra $A \ot H$ with the following algebra structures:
\begin{eqnarray*}
(a \ot x) (b\ot y) = \bigl(a b_{(1)}  \ot (x \triangleleft b_{(2)}) y \bigl), \qquad(a \ot x) \circ (b \ot y) = \bigl(a b \ot y x \bigl)
\end{eqnarray*}
is a Hopf brace. If $A$ is cocommutative as well we obtain a cocommutative Hopf brace on $A \ot H$ and a solution of the braid equation given by:
\begin{eqnarray*} 
c \Bigl((a \ot x) \ot (b \ot y) \Bigl) = \Bigl(b_{(1)} \ot \bigl( S(x_{(1)}  \triangleleft b_{(2)} \bigl) y x_{(2)} \Bigl) \ot \Bigl(T(b_{(3)}) a b_{(5)} \ot T\bigl(S(x_{(3)}) \triangleleft b_{(4)} \bigl)\Bigl)
\end{eqnarray*}
\end{examples}

\begin{example}
Let $\mathfrak{g}$ be a finite dimensional Lie algebra with $k$-basis $\{x_{1},\, x_{2},\, \cdots,\, x_{n}\}$. Then the adjoint representation of $\mathfrak{g}$ uniquely induces a left $U(\mathfrak{g})$-module bialgebra structure on $U(\mathfrak{g})$ given as follows for all $i$, $j \in 1, 2, \cdots, n$ and $k \in \NN^{*}$:
\begin{eqnarray*}
x_{i}  \triangleright x_{j}^{k} =  \sum_{t=0}^{k-1} x_{j}^{t} (x_{i}x_{j} - x_{j}x_{i}) x_{j}^{k-1-t}.
\end{eqnarray*}
Hence, we obtain a cocommutative Hopf brace on $U(\mathfrak{g}) \ot U(\mathfrak{g})$ as in \exref{gen}, $1)$.
\end{example}

\section{On the category of Hopf braces}\selabel{3}

In this section we consider some category-theoretic properties of  Hopf braces. More precisely, we will prove that
the category of Hopf braces (resp. cocommutative Hopf braces) is complete. This allows for the construction of new Hopf braces.
Throughout this section it is convenient to denote the two algebra structures of a Hopf brace $H$ by $m$,
respectively $\overline{m}$. Using this notation, \equref{1.1} can be written equivalently as follows:
\begin{equation}\eqlabel{0.11}
\overline{m}({\rm I} \ot m) = m(m \ot {\rm I}) (\overline{m} \ot S \ot \overline{m})({\rm I} \ot \tau \ot {\rm I} \ot {\rm I})({\rm I} \ot {\rm I}  \ot \tau \ot {\rm I})(\Delta \ot {\rm I} \ot {\rm I} \ot {\rm I})(\Delta \ot {\rm I} \ot {\rm I})
\end{equation}
where $\tau: H \ot H \to H \ot H$ denotes the flip map, i.e., $\tau(x \ot y) = y \ot x$ for all $x$, $y \in H$.

\begin{theorem}\thlabel{limit}
The category $k$-HopfBr is complete, i.e., it has all small limits.
\end{theorem}
\begin{proof}
As the construction of limits in the category of Hopf braces is
rather cumbersome we will restrict ourselves to products and
equalizers; this ensures the existence of all small limits in the
aforementioned category. We start by constructing products. To
this end, let $\bigl(H_{i}, m_{i}, \eta_{i}, \overline{m}_{i},
u_{i}, \Delta_{i}, \varepsilon_{i}, S_{i}, T_{i}\bigl)_{i \in I}$
be a family of Hopf braces, i.e., $\bigl(H_{i}, m_{i}, \eta_{i},
\Delta_{i}, \varepsilon_{i}, S_{i}, \bigl)_{i \in I}$ and
respectively $\bigl(H_{i}, \overline{m}_{i}, u_{i}, \Delta_{i},
\varepsilon_{i}, T_{i}\bigl)_{i \in I}$ are Hopf algebras such
that the compatibility condition \equref{0.11} is fulfilled.
Consider $\bigl((H, \Delta, \varepsilon),\, (\pi_{i})_{i \in
I}\bigl)$ to be the product (in $k$-CoAlg) of the underlying
family of coalgebras (see \cite[Theorem 1.1]{A1} for the explicit
construction), where $\pi_{i} : H \to H_{i}$ are coalgebra maps
for all $i \in I$. Since each $m_{i}$, $\eta_{i}$,
$\overline{m}_{i}$, $u_{i}$ is a coalgebra map, the universality
of the coproduct in $k$-CoAlg yields unique coalgebra maps $\eta$,
$u: k \to H$ and $m$, $\overline{m}: H \ot H \to H$ such that the
following diagrams are commutative for all $i \in I$:
\begin{eqnarray*}
\hspace*{-4mm}\xymatrix {& k\ar[d]_{\eta}\ar[dr]^{\eta_{i}}\\
 & {H} \ar[r]_{\pi_{i}} & H_{i}}\,
 \xymatrix {& k\ar[d]_{u}\ar[dr]^{u_{i}}\\
 & {H} \ar[r]_{\pi_{i}} & H_{i}}\,
 \xymatrix {& H \ot H \ar[d]_{m}\ar[dr]^{m_{i}(\pi_{i} \ot \pi_{i})}\\
 & {H} \ar[r]_{\pi_{i}} & H_{i}}\,
  \xymatrix {& H \ot H \ar[d]_{\overline{m}}\ar[dr]^{\overline{m}_{i}(\pi_{i} \ot \pi_{i})}\\
 & {H} \ar[r]_{\pi_{i}} & H_{i}}
\end{eqnarray*}
\begin{eqnarray}
{\rm i.e.,} \,\,\, && \pi_{i} \eta = \eta_{i}\eqlabel{0.2.0},\,\,\,  \pi_{i} u = u_{i}\\\
&& \pi_{i} m  = m_{i} (\pi_{i} \ot \pi_{i}) \eqlabel{0.2}\\
&& \pi_{i} \overline{m}  = \overline{m}_{i} (\pi_{i} \ot \pi_{i}) \eqlabel{0.2.4}
\end{eqnarray}
Since proving that $\bigl(H, m, \eta, \Delta, \varepsilon \bigl)$ and respectively
$\bigl(H, \overline{m}, u, \Delta, \varepsilon \bigl)$ are actually bialgebras goes essentially
in the same vain as the proof of \cite[Theorem 1.5]{A1} we will not include the details here;
we refer the reader to \cite{A1}.

Next we construct antipodes for the two bialgebras $\bigl(H, m, \eta, \Delta, \varepsilon \bigl)$
and $\bigl(H, \overline{m}, u, \Delta, \varepsilon \bigl)$. Using again the universality of the
product in the category $k$-CoAlg we obtain two unique coalgebra morphisms $S$, $T: H \to H^{op,\,cop}$
such that the following diagrams are commutative for all $i \in I$:
\begin{eqnarray*}
\hspace*{-4mm}\xymatrix {& H^{op,\,cop}\ar[rr]^{\pi_{i}}\ar[d]_{S} & {} &{H_{i}^{op,\,cop}}\ar[d]^{S_{i}}\\
 & {H} \ar[rr]_{\pi_{i}} & {} & H_{i}}\,
\xymatrix {& H^{op,\,cop}\ar[rr]^{\pi_{i}}\ar[d]_{T} & {} &{H_{i}^{op,\,cop}}\ar[d]^{T_{i}}\\
 & {H} \ar[rr]_{\pi_{i}} & {} & H_{i}}
\end{eqnarray*}
\begin{eqnarray}
{\rm i.e.,} \,\,\,\, \pi_{i} S = S_{i} \pi_{i} \,\,\, {\rm and}
\,\,\, \pi_{i} T = T_{i} \pi_{i}\eqlabel{0.3.0}
\end{eqnarray}
In order to prove that $S$, $T: H^{op,cop} \to H$ are also algebra
maps it will suffice to show that for all $i \in I$ we have
$\pi_{i} m (S \ot S)\,\tau = \pi_{i} S m$, and respectively
$\pi_{i} \overline{m} (T \ot T)\, \tau= \pi_{i} T \overline{m}$.
Indeed, having in mind that $S_{i}: H_{i}^{op,cop} \to H_{i}$ is
an algebra map for all $i \in I$ we obtain:
\begin{eqnarray*}
\underline{\pi_{i} S} m &\stackrel{\equref{0.3.0}} {=}& S_{i}\underline{ \pi_{i}  m}  \stackrel{\equref{0.2}} {=} \underline{S_{i} m_{i}} (\pi_{i} \ot \pi_{i}) = m_{i} (S_{i} \ot S_{i})\, \underline{\tau (\pi_{i} \ot \pi_{i})}\\
&=& m_{i} (S_{i} \ot S_{i}) (\pi_{i} \ot \pi_{i})\, \tau = m_{i} (\underline{S_{i} \pi_{i}}  \ot \underline{S_{i} \pi_{i}})\, \tau\\
&\stackrel{\equref{0.3.0}} {=}& m_{i} (\underline{\pi_{i} S  \ot
\pi_{i} S})\, \tau = \underline{m_{i} (\pi_{i} \ot \pi_{i})} (S  \ot S)\, \tau\\
&\stackrel{\equref{0.2}} {=}&  \pi_{i} m (S  \ot S)\, \tau
\end{eqnarray*}
as desired. A similar computation shows that $T$ is also an algebra map.

Now exactly as in the proof of \cite{A1} one can prove that $\bigl(H, m, \eta, \Delta, \varepsilon, S, \bigl)$
and respectively $\bigl(H, \overline{m}, u, \Delta, \varepsilon, T\bigl)$ are actually Hopf algebras. We are
left to prove that \equref{0.11} holds. As before, it will suffice to show that for all $i \in I$ the following
compatibility holds true:
$$
\pi _{i} \overline{m}({\rm I} \ot m) = \pi_{i} m(m \ot {\rm I}) (\overline{m} \ot S \ot \overline{m})({\rm I} \ot \tau \ot {\rm I} \ot {\rm I})({\rm I} \ot {\rm I}  \ot \tau \ot {\rm I})(\Delta \ot {\rm I} \ot {\rm I} \ot {\rm I})(\Delta \ot {\rm I} \ot {\rm I})
$$
Indeed, for all $i \in I$, we have:
\begin{eqnarray*}
&&\hspace*{-15mm} \underline{\pi_{i} m}(m \ot {\rm I}) (\overline{m} \ot S \ot \overline{m})({\rm I} \ot \tau \ot {\rm I} \ot {\rm I})({\rm I} \ot {\rm I}  \ot \tau \ot {\rm I})(\Delta \ot {\rm I} \ot {\rm I} \ot {\rm I}) (\Delta \ot {\rm I} \ot {\rm I}) \\
&\stackrel{\equref{0.2}} {=}&m_{i} \underline{(\pi_{i} \ot \pi_{i}) (m \ot {\rm I})} (\overline{m} \ot S \ot \overline{m})({\rm I} \ot \tau \ot {\rm I} \ot {\rm I})({\rm I} \ot {\rm I}  \ot \tau \ot {\rm I})\\
&&(\Delta \ot {\rm I} \ot {\rm I} \ot {\rm I})(\Delta \ot {\rm I} \ot {\rm I}) \\
&=& m_{i} (\underline{\pi_{i} m} \ot \pi_{i}) (\overline{m} \ot S \ot \overline{m})({\rm I} \ot \tau \ot {\rm I} \ot {\rm I})({\rm I} \ot {\rm I}  \ot \tau \ot {\rm I})(\Delta \ot {\rm I} \ot {\rm I} \ot {\rm I})\\
&&(\Delta \ot {\rm I} \ot {\rm I}) \\
&\stackrel{\equref{0.2}} {=}& m_{i} \bigl(\underline{m_{i} (\pi_{i} \ot \pi_{i}) \ot \pi_{i}}\bigl) (\overline{m} \ot S \ot \overline{m})({\rm I} \ot \tau \ot {\rm I} \ot {\rm I})({\rm I} \ot {\rm I}  \ot \tau \ot {\rm I})\\
\end{eqnarray*}
\begin{eqnarray*}
&&(\Delta \ot {\rm I} \ot {\rm I} \ot {\rm I})(\Delta \ot {\rm I} \ot {\rm I}) \\
&=& m_{i} (m_{i} \ot {\rm I})\underline{(\pi_{i} \ot \pi_{i} \ot \pi_{i}) (\overline{m} \ot S \ot \overline{m})}({\rm I} \ot \tau \ot {\rm I} \ot {\rm I})({\rm I} \ot {\rm I}  \ot \tau \ot {\rm I})\\
&&(\Delta \ot {\rm I} \ot {\rm I} \ot {\rm I})(\Delta \ot {\rm I} \ot {\rm I}) \\
&=& m_{i} (m_{i} \ot {\rm I})(\underline{\pi_{i} \overline{m}} \ot \underline{\pi_{i} S} \ot \underline{\pi_{i} \overline{m}})({\rm I} \ot \tau \ot {\rm I} \ot {\rm I})({\rm I} \ot {\rm I}  \ot \tau \ot {\rm I})\\
&&(\Delta \ot {\rm I} \ot {\rm I} \ot {\rm I})(\Delta \ot {\rm I} \ot {\rm I}) \\
&\stackrel{\equref{0.2.4}, \equref{0.3.0}} {=}& m_{i} (m_{i} \ot {\rm I}) \bigl(\underline{\overline{m}(\pi_{i} \ot \pi_{i}) \ot S_{i} \pi_{i} \ot \overline{m}(\pi_{i} \ot \pi_{i})} \bigl)({\rm I} \ot \tau \ot {\rm I} \ot {\rm I})\\
&& ({\rm I} \ot {\rm I}  \ot \tau \ot {\rm I})(\Delta \ot {\rm I} \ot {\rm I} \ot {\rm I})(\Delta \ot {\rm I} \ot {\rm I}) \\
&=& m_{i} (m_{i} \ot {\rm I}) (\overline{m}\ot S_{i} \ot \overline{m}) (\pi_{i} \ot \pi_{i} \ot \pi_{i} \ot \pi_{i} \ot \pi_{i})({\rm I} \ot \tau \ot {\rm I} \ot {\rm I})\\
&& ({\rm I} \ot {\rm I}  \ot \tau \ot {\rm I})(\Delta \ot {\rm I} \ot {\rm I} \ot {\rm I})(\Delta \ot {\rm I} \ot {\rm I}) \\
&=& m_{i} (m_{i} \ot {\rm I}) (\overline{m}\ot S_{i} \ot \overline{m})({\rm I} \ot \tau \ot {\rm I} \ot {\rm I})({\rm I} \ot {\rm I}  \ot \tau \ot {\rm I})\\
&& \bigl(\underline{(\pi_{i} \ot \pi_{i})\Delta}  \ot \pi_{i} \ot \pi_{i} \ot \pi_{i}\bigl)(\Delta \ot {\rm I} \ot {\rm I}) \\
&=& m_{i} (m_{i} \ot {\rm I}) (\overline{m}\ot S_{i} \ot \overline{m})({\rm I} \ot \tau \ot {\rm I} \ot {\rm I})({\rm I} \ot {\rm I}  \ot \tau \ot {\rm I})\\
&& \bigl(\Delta_{i} \pi_{i}  \ot \pi_{i} \ot \pi_{i} \ot \pi_{i}\bigl)(\Delta \ot {\rm I} \ot {\rm I}) \\
&=& m_{i} (m_{i} \ot {\rm I}) (\overline{m}\ot S_{i} \ot \overline{m})({\rm I} \ot \tau \ot {\rm I} \ot {\rm I})({\rm I} \ot {\rm I}  \ot \tau \ot {\rm I})\\
&& (\Delta_{i} \ot I \ot I \ot I )\bigl(\underline{(\pi_{i} \ot \pi_{i})\Delta} \ot \pi_{i} \ot \pi_{i}\bigl) \\
&=& \underline{m_{i} (m_{i} \ot {\rm I}) (\overline{m}\ot S_{i} \ot \overline{m})({\rm I} \ot \tau \ot {\rm I} \ot {\rm I})({\rm I} \ot {\rm I}  \ot \tau \ot {\rm I})}\\
&& \underline{(\Delta_{i} \ot I \ot I \ot I )(\Delta_{i} \ot {\rm I} \ot {\rm I})} (\pi_{i} \ot \pi_{i} \ot \pi_{i}) \\
&\stackrel{H_{i} - {\rm Hopf\,\, brace}} {=}& \overline{m}_{i} ({\rm I }\ot m_{i})(\pi_{i} \ot \pi_{i} \ot \pi_{i})\\
&=& \overline{m}_{i} \bigl(\pi_{i} \ot \underline{m_{i}(\pi_{i} \ot \pi_{i}})\bigl)\\
&\stackrel{\equref{0.2}} {=}& \overline{m}_{i} (\pi_{i} \ot \pi_{i} m)\\
&=&  \overline{m}_{i} \underline{(\pi_{i} \ot \pi_{i})({\rm I} \ot m)}\\
&\stackrel{\equref{0.2.4}} {=}& \underline{\pi _{i} \overline{m}}({\rm I} \ot m)
\end{eqnarray*}

Now let $f$, $g: A \to B$ be two morphisms of Hopf braces.
Consider $S = \{a \in A ~|~ f(a) = g(a)\}$ and let $C$ be the
sum of all subcoalgebras of $A$ contained in $S$. It can be
easily seen (see the proof of \cite[Theorem 1.7]{A1}) that $C$ is
actually a Hopf subalgebra with respect to both Hopf algebra
structures of $A$ and is obviously a Hopf brace. Thus, $(C,\,i)$
is the equalizer of the pair of morphisms $(f,\,g)$ in $k$-HopfBr,
where $i: C \to A$ is the canonical inclusion.
\end{proof}

\begin{example}
Let $\bigl(H_{i}, m_{i}, \eta_{i}, \overline{m}_{i},
u_{i}, \Delta_{i}, \varepsilon_{i}, S_{i}, T_{i}\bigl)_{i=1,2}$ be two Hopf braces and $f$, $g: \bigl(H_{1}, m_{1}, \eta_{1}, \overline{m}_{1},
u_{1}, \Delta_{1}, \varepsilon_{1}, S_{1}, T_{1}\bigl) \to \bigl(H_{2}, m_{2}, \eta_{2}, \overline{m}_{2},
u_{2}, \Delta_{2}, \varepsilon_{2}, S_{2}, T_{2}\bigl)$ homomorphisms of Hopf braces. Then 
$$
H' = \{h \in H_{1} ~|~ h_{(1)} \ot f(h_{(2)}) \ot h_{(3)} = h_{(1)} \ot g(h_{(2)}) \ot h_{(3)}\}
$$
is a Hopf brace. Indeed, it can be easily seen that $(H',\,i)$ is the equalizer of $(f,\,g)$ where $i: H' \to H_{1}$ is the inclusion homomorphism, and the desired conclusion follows by \thref{limit}.
\end{example}

\begin{corollary}
The category $k$-(co)HopfBr is complete, i.e., it has all limits.
\end{corollary}
\begin{proof}
The constructions are similar to those in \thref{limit}. Indeed, we only point out that the product in the category of cocommutative coalgebras is simply the tensor product of the given family of coalgebras. Therefore, the product of a given family of cocommutative Hopf braces is the tensor product of the underlying coalgebras with the Hopf brace structure defined exactly as in the proof of \thref{limit}. 
\end{proof}

In \cite[Theorem 1.12]{AGV} it was proven that given a Hopf algebra $A$ the category $\mathfrak{C}(A)$ is equivalent to the full subcategory of $k$-HopfBr whose objects are all Hopf braces whose first Hopf algebra structure is that of $A$. The aforementioned result can be easily extended to an equivalence between  $\mathfrak{C}$ and $k$-HopfBr:

\begin{theorem}
The categories $k$-HopfBr and $\mathfrak{C}$ are equivalent. In particular, the category of bijective $1$-cocycles $\mathfrak{C}$ is complete.
\end{theorem}
\begin{proof}
By \cite[Lemma 1.8]{AGV}, any Hopf brace $(A, m, \eta, \overline{m},
u, \Delta, \varepsilon, S, T)$ induces a left $(A, \overline{m},
u,$ $\Delta, \varepsilon, T)$-module algebra structure on $(A, m, \eta, \Delta, \varepsilon, S)$.
Then the functor $F$: $k$-HopfBr$\to \mathfrak{C}$ defined by:
\begin{eqnarray*}
&&F(A, m, \eta, \overline{m},
u, \Delta, \varepsilon, S, T) = \bigl(1_{A}: (A, \overline{m},
u, \Delta, \varepsilon, T)  \to (A, m, \eta, \Delta, \varepsilon, S)\bigl)\\
&&F(f) = (f,\,f) 
\end{eqnarray*}
for all Hopf braces $(A, m, \eta, \overline{m},
u, \Delta, \varepsilon, S, T)$ and all Hopf brace homomorphisms $f: (A, m_{A}, \eta_{A}, \overline{m}_{A},
u_{A}, \Delta_{A}, \varepsilon_{A}, S_{A}, T_{A}) \to (B, m_{B}, \eta_{B}, \overline{m}_{B},
u_{B}, \Delta_{B}, \varepsilon_{B}, S_{B}, T_{B})$,
provides the claimed equivalence of categories. This can be seen by following precisely the same steps as in the proof of \cite[Theorem 1.12]{AGV}. Hence, in particular, \thref{limit} implies that the category of bijective $1$-cocycles is complete as well.
\end{proof}

\end{document}